\documentclass{article}
\usepackage{amsmath, amsfonts, amssymb, verbatim}
\usepackage{multirow} 
\usepackage{pstricks}
\usepackage{auto-pst-pdf}
\usepackage[all,cmtip]{xy}
\usepackage{lscape}
\usepackage{longtable}
\usepackage{hyperref}

\newtheorem{theorem}{Theorem}[section]
\newtheorem{definition}[theorem]{Definition}

\newtheorem{lemma}[theorem]{Lemma}

\newtheorem{proposition}[theorem]{Proposition}

\newenvironment{proof}[1][]{ \textbf{Proof#1. }}{$\Box$\medskip}
\newcommand{\tpi}{\tilde{\pi}}

\newcommand{\ch}{{\,\vee}}
\newcommand{\gog}{{\mathfrak g}}

\newcommand{\gop}{{\mathfrak p}}

\newcommand{\goh}{{\mathfrak h}}

\newcommand{\gon}{{\mathfrak n}}
\newcommand{\gol}{{\mathfrak l}}
\newcommand{\gob}{{\mathfrak b}}

\newcommand{\ad}{\mathrm{ad}}
\newcommand{\gr}{\mathrm{gr}}

\DeclareMathOperator{\pr}{pr}

\newcommand{\scalarSA}[2]{\langle #1,#2 \rangle_{\gog'}}
\newcommand{\scalarLA}[2]{\langle #1,#2 \rangle_{\gog}}
\newcommand\clplus{\hbox{$\subset${\raise0.3ex\hbox{\kern -0.55em ${\scriptscriptstyle +}$}}\ }}
\newcommand\crplus{\hbox{$\supset${\raise1.05pt\hbox{\kern -0.60em ${\scriptscriptstyle +}$}}\ }}

\newcommand{\mV}{\mathbb V}

\newcommand{\mC}{\mathbb C}

\newcommand{\mZ}{\mathbb Z}

\newcommand{\mN}{\mathbb N}

\newcommand{\lC}{\cal{C}}
\newcommand{\lD}{\cal D}
\newcommand{\lU}{\cal U}

\newcommand{\mW}{\mathbb W}

\newcommand{\pa}{\partial}
\newcommand{\px}[1]{\frac{\partial}{\partial x_{#1}}}

\newcommand{\End}{\operatorname{End}}
\newcommand{\Pol}{\operatorname{Pol}}
\newcommand{\Diff}{\operatorname{Diff}}

\newcommand{\Ind}{\operatorname{Ind}}

\newcommand{\LieGtwo}{\mathrm{Lie~}G_2}
\newcommand{\LieAlgPair}[2]{#1\stackrel{i}{\hookrightarrow} #2}

\begin{document}

\title{The F-method and a branching problem for generalized Verma modules associated to $({\LieGtwo},{so(7)})$}
\author{Todor Milev and Petr Somberg}

\date{}

\maketitle

\begin{abstract} 
The branching problem for a couple of non-compatible Lie algebras and their parabolic
subalgebras applied to generalized Verma modules was recently discussed in \cite{ms}. 
In the present article, we employ the recently developed F-method, \cite{KOSS1}, \cite{KOSS2} to the couple of non-compatible Lie algebras $\mathrm{Lie~}G_2\stackrel{i}{\hookrightarrow}{so(7)}$,
and generalized conformal ${so(7)}$-Verma modules of scalar type.
As a result, we classify the $i(\LieGtwo) \cap \gop$-singular vectors for this class of $so(7)$-modules. 
\end{abstract}
 
%\tableofcontents 



%%%%%%%%%%%%%%%%%%%%%%%%%%%%%%%%%%%%%%%%%%%%%%%%%%%%%%%%%%%%%%%%%%%%%%%%%%%%%%%%%%%%%%%%%%%%

\section{Introduction and Motivation}

The subject of our article has its motivation in the Lie theory problem of 
branching rules for finite dimensional simple Lie algebras and composition structure of generalized Verma modules, 
and dually in the geometrical problem related to the construction of invariant
differential operators in parabolic invariant theories.

We assume that $\gog, \gog'$ are complex semisimple Lie algebras
and $i:\gog'\hookrightarrow\gog$ is an injective homomorphism. Then
$i(\gog')$ is (complex) reductive in $\gog$ and we can choose Borel subalgebras  $\gob'\subset\gog'$ and $\gob \subset\gog$ such that  $i(\gob')\subset\gob$. Let ${\gop\supset \gob}$ be a parabolic subalgebra of $\gog$.  Let $M^\gog_\gop(\mV_{\lambda})$ be the generalized Verma $\gog$-module induced from the irreducible finite dimensional 
$\gop$-module $\mV_{\lambda}$ with highest weight $\lambda$.
We define the branching problem for $M^\gog_\gop(\mV_{\lambda})$ over 
${\gog'}$ to be the problem of finding all $\gob'$-singular 
vectors in $M^\gog_\gop(\mV_{\lambda})$, that is, the set of 
all vectors annihilated by image of the nilradical of $\gob'$
on which the image of the Cartan subalgebra of $\gob'$ has diagonal action. 

In the recent article \cite{ms}, under certain technical assumptions, 
we proved that $M^\gog_\gop(\mV_{\lambda})$ has (finite or infinite) Jordan-H\"older 
series over ${\gog'}$, and enumerated the $\gob'$-highest weights 
$\mu$ appearing in the series. We also computed the dimension  $m(\mu,\lambda)$ of the vector space of $\gob'$-highest weights of weight $\mu$  as a function of $\mu$ and $\lambda$. Further we gave a procedure for producing explicit formulas for some (but not all) $\gob'$-highest weight vectors.

As an example, we discussed $\LieAlgPair{\LieGtwo} {so(7)}$. Restricting our attention to the parabolic subalgebra $\gop\simeq\gop_{(1,0,0)}$ and the 6 infinite families of highest weights $x\varepsilon_1$, $x\varepsilon_1+\omega_2$,  $x\varepsilon_1+\omega_3$, 
$x\varepsilon_1+2\omega_2$,  
$x\varepsilon_1+\omega_1+\omega_2$,  
$x\varepsilon_1+2\omega_3$ we computed in \cite{ms} all $\gob'$-singular 
vectors with $\gob'$-dominant weights. From the theory of generalized Verma modules we know that, depending on the integrality and dominance of the $\gob$-highest weight, $M^\gog_\gop(\mV_{\lambda})$ has $\gob$-singular (and therefore $\gob'$-singular) vectors other than the highest weight vector. 
Therefore these vectors give additional $\gob'$-singular vectors whose weights are not $\gob'$-dominant (and are not computed in \cite{ms}). 

Fix the pair $\LieAlgPair{\LieGtwo}{ so(7)}$ and fix the parabolic subalgebra to be the parabolic subalgebra $\gop_{(1,0,0)}\subset so(7)$ obtained by crossing out the first (long) root of $so(7)$. Let 
\[
\gop'_{(1,0)}=i^{-1}(\LieGtwo).
\]  
In the present article, for the family of $so(7)$-highest weights of the form $x\varepsilon_1$,  we prove that if $x \in \{-3/2,-1/2,1/2,\dots\}$, the module $M^\gog_{\gop_{(1,0,0)}}(\mV_{x\varepsilon_1})$ has, besides its highest weight vector, exactly one $\gop'_{(1,0)}$-singular vector, and has no $\gop'_{(1,0)}$-singular vectors otherwise. Here we recall that, for an arbitrary parabolic subalgebra $\gop'$, a $\gop'$-singular vector is defined as a vector that is annihilated by all elements of the Levi part of $\gop'$, and therefore has weight that projects to zero onto the Levi part of $\gop'$ (``weight of scalar type''). Our result has a somewhat unusually sounding consequence: the $\gop'_{(1,0)}$-singular vector in  $M^{so(7)}_{\gop_{(1,0,0)}}(\mV_{x\varepsilon_1})$  must automatically be $\gop_{(1,0,0)}$-singular. This fact must necessarily fail to generalize for sufficiently large values of $a,b$ and highest weights of the form $\lambda=x\omega_1+a\omega_2+b\omega_3 $. Indeed, the number of $\gop_{(1,0,0)}$-singular vectors in $M^{so(7)}_{\gop_{(1,0,0)}}(\mV_{\lambda})$ is uniformly bounded, while the number $m(x\omega_1, x\omega_1+ a\omega_2 +b\omega_3)$ grows as a linear function of $a$ and $b$.

We would like to note that our example goes beyond the compatible couples of Lie algebras discussed in \cite{KOSS1}, \cite{KOSS2}.

A geometric motivation for the branching problem can be described as follows.
Let $G, G'$ be the connected and simply connected Lie groups with 
Lie algebras $\gog, {\gog}'$.
Let $P$ be the parabolic subgroup of $G$ with Lie algebra $\gop$, and 
let $L\subset P$ be its Levi factor.
Then there is a well-known
equivalence between invariant differential 
operators acting on induced
representations and homomorphisms of generalized Verma modules,
realized by the natural pairing
\begin{eqnarray}\label{eqInvriantDiffOpActing}
Ind^G_P(\mathbb V_\lambda(L)^*)\times M^\gog_\gop(\mV_{\lambda})\longrightarrow \mC,
\end{eqnarray}
where $\mathbb V_\lambda(L)$ denotes the finite-dimensional irreducible $L$-module, 
$\mathbb V_\lambda(L)^*$ is its dual, 
and $Ind^G_P$ denotes induction from $P$ to $G$. 
As a consequence, the singular vectors
constructed in the article determine invariant differential operators
acting between induced representations of $i({G}')$.
It is quite interesting to construct these invariant differential
operators, in particular their curved extensions as lifts
to homomorphisms of semiholonomic generalized Verma modules.

Our motivation for the particular example
$\LieAlgPair{\LieGtwo}{so(7)}$ comes from a natural problem in conformal
geometry of dimension $5$ (note that $so(7)$ is the complexification of the conformal Lie algebra in dimension $5$), 
see \cite{gw} and references therein.
A geometrical characterization of the reduction of
the structure group with Lie algebra $so(7)$ down to $\LieGtwo$ for a given inducing representation
$\mV_\lambda$ is then given by invariant differential operators acting on sections of the
associated vector bundles, intertwined by actions of $so(7)$ and ${\LieGtwo}$.

The structure of the article is as follows. In Section $2$
we recall basic conventions on $so(7), \LieGtwo, i(\LieGtwo)$ and the structure of their parabolic 
subalgebras relative to the embedding $i$. In Section 3, we use \eqref{eqInvriantDiffOpActing} to transform the problem 
of finding differential invariants for
$(so(7),\LieGtwo),\mV_\lambda$ into an algebraic question
about homomorphisms between generalized Verma modules, corresponding
to solutions of the branching problem. In Section 4 we fix the conformal parabolic subalgebra to be  $\gop_{(1,0,0)}\subset so(7)$. Therefore by Lemma \ref{leParabolicsG2inParabolicsB3} the subalgebra $\gop'$ is given by $i(\gop')=i(\gog')\cap\gop$ and therefore equals the subalgebra $\gop'_{(1,0)}$ obtained by crossing out the first root of $\LieGtwo$. We note that $\gop'_{(1,0)}$ is not compatible $(\gog,\gop)$. We further fix the highest weight to be $\lambda\varepsilon_1$ (here we use $\lambda$ as a scalar). We then 
apply the distribution Fourier transform (the ``F-method'') 
developed in \cite{KOSS1}, \cite{KOSS2} to obtain our main result Theorem \ref{thMain}.

%%%%%%%%%%%%%%%%%%%%%%%%%%%%%%%%%%%%%%%%%%%%%%%%%%%%%%%%%%%%%%%%%%%%%%%%%%%%%%%%%%%%%%%%%%%%%%%%%%%%%%%

\section{Branching problem and (non-compatible) parabolic subalgebras 
for the pair $\LieAlgPair{\LieGtwo}{so(7)}$}
\label{secso7G2LieAlgstructure}\label{secG2inB3def}

In the present section we introduce the Lie theoretic conventions for the complex Lie algebra 
$so(7)$, exceptional Lie algebra $\LieGtwo$, and Levi resp. parabolic subalgebras $\gop$ of 
$so(7)$ relative to parabolic subalgebras $i(\gop')$ of $i(\LieGtwo)$. These will 
be used in the subsequent Section $3$, where we employ the F-method. For more detailed
review, cf. \cite{ms}.

We start by fixing a Chevalley-Weyl basis of the Lie algebra $so(2n+1)$. 
Let the defining vector space $V$ of $so(2n+1)$ have a basis 
$e_{1},\dots e_{n}, e_0, e_{-1}, \dots e_{-n}$, where the defining symmetric bilinear 
form $B$ of $so(2n+1)$ is given by $ B(e_i, e_j):=0, i\neq -j$, 
$ B(e_i,e_{-i}):=1$, $B(e_i,e_0):=0$, $B(e_0,e_0):=1$, or alternatively defined as an 
element of $S^2(V^*)$,
\begin{equation}\label{eqDefiningSymmetricBilinearFormTypeB}
B:=\sum_{i=-n}^{n} e_i^*\otimes e_{-i}^*=(e_0^*)^2+ 2\sum_{i=1}^{n} e_i^* e_{-i}^*,
\end{equation}
under the identification $v^*w^*:=\frac {1}{2!} \left(v^*\otimes w^*+w^*\otimes v^*\right)$.

In the basis $e_{1},\dots e_{n}, e_0, e_{-1}, \dots e_{-n}$, the 
matrices of the elements of $so(2n+1)$ are of the form
\[
\left(\begin{array}{c|c|c}A&\begin{array}{c}v_1\\ \vdots \\ v_n\end{array} &C=-C^T 
\\\hline \begin{array}{ccc}w_1 &\dots&  w_n\end{array} &0& \begin{array}{ccc}-v_1 
&\dots&  -v_n\end{array} \\\hline D=-D^T&\begin{array}{c}-w_1\\ \vdots \\ -w_n\end{array} 
& -A^T\end{array}\right), \]
i.e., all matrices $\mathbf C$ such that $\mathbf A^t\mathbf B+\mathbf B\mathbf A=0$. 
We fix $e_{1}^*,\dots e_{n}^*, e_0^*, e_{-1}^*, \dots e_{-n}^*$ to be basis of $V^*$ 
dual to $e_{1},\dots e_{n}, e_0, e_{-1}, \dots e_{-n}$. We identify elements of $End(V)$ 
with elements of $V\otimes V^*$.
 In turn, we identify elements of $End(V)$ with their matrices in the basis 
 $e_{1},\dots, e_{n}$, $e_0$, $e_{-1}, \dots, e_{-n}$.

Fix the Cartan subalgebra $\goh$ of $so(2n+1)$ to be the subalgebra of
diagonal matrices, i.e., the subalgebra spanned by the vectors 
$e_i\otimes e_i^*-e_{-i}\otimes e_{-i}^* $. 
Then the basis vectors $e_{1},\dots e_{n}, e_0, e_{-1}, \dots e_{-n}$ are a basis for 
the $\goh$-weight vector decomposition of $V$. Let the $\goh$-weight of $e_i, i>0$, be 
$\varepsilon_i$. Then the $\goh$-weight of $e_{-i}, i>0$ is $-\varepsilon_i$, and an 
$\goh$-weight decomposition of $so(2n+1)$ is given by the elements 
$g_{\varepsilon_i-\varepsilon_j}:= e_i\otimes e_j^*-e_{-j}\otimes e_{-i}^*$, 
$g_{\pm(\varepsilon_i+\varepsilon_j)}:=e_{\pm i}\otimes e_{\mp j}^*-e_{\pm j}\otimes e_{\mp i}^*$ 
and $g_{\pm \varepsilon_i}:=\sqrt{2}\left(e_{\pm i}\otimes e_0^*-e_0\otimes e_{\mp i}^*\right)$, 
where $i,j>0$.

Define the symmetric bilinear form $\scalarLA{\bullet}{\bullet}$ on $\goh^*$ by 
$\scalarLA{\varepsilon_i}{\varepsilon_j}=1$ if $i=j$ and zero otherwise. 

The root system of $so(2n+1)$ with respect to $\goh$ is given by $\Delta(\gog):=\Delta^+(\gog)\cup\Delta^-(\gog)$, 
where we define 
\begin{equation}\label{eqPosRootSystemB3}
\Delta^+(\gog):=\{\varepsilon_i\pm\varepsilon_j|1\leq i< j\leq n\}\cup\{\varepsilon_i|1\leq i\leq n\}
\end{equation}
and $\Delta^-(\gog):=-\Delta^{+}(\gog)$. We fix the Borel subalgebra $\gob$ of $so(2n+1)$ to be the 
subalgebra spanned by $\goh$ and the elements $g_{\alpha}, \alpha\in \Delta^+(\gog)$. The simple positive 
roots corresponding to $\gob$ are then given by 
\[
\eta_1:=\varepsilon_1-\varepsilon_2,\dots,  \eta_{n-1}:=\varepsilon_{n-1}-\varepsilon_{n},\eta_{n}:=\varepsilon_{n}\quad .
\]

For the remainder of this Section we fix the odd orthogonal Lie algebra to be $so(7)$. We order the 18 roots of $so(7)$ in graded lexicographic order with respect to their simple basis coordinates. We then label the negative roots by the indices $-9, \dots, -1$ and the positive roots by the indices $1,\dots, 9 $. Finally, we abbreviate the Chevalley-Weyl generator $g_{\alpha}\in so(7)$ by $g_i$, where $i$ is the label of the corresponding root. For example, $g_{\pm 1}=g_{\pm(\varepsilon_1-\varepsilon_2)}$, $g_{\pm 2}=g_{\pm (\varepsilon_2-\varepsilon_3)}$, $g_{\pm 3}=g_{\pm (\varepsilon_3)}$ are the simple positive and negative generators,  the element $g_{-9}=g_{-\varepsilon_1-\varepsilon_2}$ is the Chevalley-Weyl generator corresponding to the lowest root, and so on. We furthermore set  $h_1:=[g_1,g_{-1}]$, $h_2:=[g_2,g_{-2}]$, $h_3:=1/2[g_{3}, g_{-3}]$. 

Let now $\gog'= \LieGtwo$. One way of defining the positive 
root system of $\LieGtwo$ is by setting it to be the set of vectors
\begin{equation}\label{eqPosRootSystemG2}
\Delta( \gog'):=\{\pm(1, 0), \pm(0, 1), \pm(1, 1), \pm(1, 2), \pm(1, 3), \pm(2, 3)\}.
\end{equation}
We set $\alpha_1:=(1,0)$ and $\alpha_2:=(0,1)$. We fix a bilinear form $\scalarSA{\bullet}{\bullet}$ on $\goh'$, proportional
%\footnote{For the bilinear form induced by the Killing form of $\LieGtwo$, 
%the long root has squared length $36$, so the coefficient of proportionality is $6$} 
to the one induced by Killing form by setting 
\begin{eqnarray}
\left(\begin{array}{cc}\scalarSA{\alpha_1}{\alpha_1} & \scalarSA{\alpha_1}{\alpha_2} \\ 
\scalarSA{\alpha_2}{\alpha_1} & \scalarSA{\alpha_2}{\alpha_2}\\\end{array}\right):= 
\left(\begin{array}{cc}2 & -3\\ -3 & 6\\\end{array}\right).
\end{eqnarray}
In an  $\scalarSA{\bullet}{\bullet}$-orthogonal basis the root system of $\LieGtwo$ is drawn in Figure \ref{figRootSystemG2}.
\begin{figure}
\begin{center}
\begin{pspicture}(-1,1)(1,-1)
\psline(0,0)(0.5,0.87)
\psline(0,0)(-0.5,0.87)
\psline(0,0)(0.5,-0.87)
\psline(0,0)(-0.5,-0.87)
\psline(0,0)(1,0)
\psline(0,0)(-1,0)
\psline(0,0)(0.57,0.32) %0.58, 0.33
\psline(0,0)(-0.57,0.32)
\psline(0,0)(0.57,-0.32)
\psline(0,0)(-0.57,-0.32)
\psline(0,0)(0,0.67)
\psline(0,0)(0,-0.67)
\rput(1,0.1){\tiny{$\alpha_2$}}
\rput(-0.62,0.38){\tiny{$\alpha_1$}}
\end{pspicture}
\end{center}
\caption{\label{figRootSystemG2}The root system of $\LieGtwo$}
\end{figure} 

Similarly to the $so(7)$ case, we order the 12 roots of $\LieGtwo$ in the graded lexicographic order with respect to their simple basis coordinates, and label the roots with the indices $-6,\dots, -1$, $1, \dots, 6$. We fix a basis for the Lie algebra $\LieGtwo$ by giving a 
set of Chevalley-Weyl generators  $ g'_{i}$, $i\in \{\pm 1, \dots \pm 6\}$, 
and by setting $h'_{1}:=[ g'_1, g'_{-1}]$, $ h'_{2}:=3[ g'_2, g'_{-2}]$. Just as in the $so(7)$ case, we ask that the generator $g'_{\pm i} $ correspond to the root space labeled by $\pm i$.

All embeddings $\LieAlgPair{\LieGtwo}{so(7)}$ are conjugate over $\mathbb C$. 
One such embedding is given via 
\[
i({g}'_{\pm 2}):= g_{\pm 2}, \quad i(g'_{\pm 1}):= g_{\pm 1}+g_{\pm 3}\quad.
\] 
As $g'_{\pm 1}, g'_{\pm 2}$ generate $\LieGtwo$, 
the preceding data determines the map $i$ 
and one can directly check it is a Lie algebra homomorphism. 
Alternatively, we can use $i({g}'_{\pm 1}), i({g}'_{\pm 2})$ to
generate a Lie subalgebra of $so(7)$, 
verify that this subalgebra is indeed 
14-dimensional  and simple, and finally use this 14-dimensional image 
to compute the structure constants of $\LieGtwo$.

We denote by $\omega_1:=\varepsilon_1$, $\omega_2:=\varepsilon_1+\varepsilon_2$ and  $\omega_3:=\frac12(\varepsilon_1+\varepsilon_2+\varepsilon_3)$ 
the fundamental weights of  $so(7)$ and by $\psi_1:=2\alpha_1+\alpha_2$, $\psi_2:=3\alpha_1+2\alpha_2$ 
the fundamental weights of $\LieGtwo$.

Let $\pr: \goh^*\to {\goh'}^*$ be the map naturally induced by $i$. Then
\begin{equation}\label{eqProjectionCartanB_3toCartanG_2}
\pr(\underbrace{\varepsilon_1-\varepsilon_2}_{\eta_1})=\pr(\underbrace{\varepsilon_3}_{\eta_3})=\alpha_1, \quad \pr(\underbrace{\varepsilon_2-\varepsilon_3}_{\eta_2})=\alpha_2 ,
\end{equation}
or equivalently
\begin{equation*}
\pr(\omega_1)=\pr(\omega_3)=\psi_1, \quad \pr(\omega_2)=\psi_2.
\end{equation*}
Conversely, $\iota: {\goh'}^* \to \goh^*$ is the map 
\begin{equation}\label{eqG2rootSystemEmbeddingInB3}
\iota(\alpha_2)=3\eta_2= 3\varepsilon_2-3\varepsilon_3, \quad \iota(\alpha_1)=\eta_1+2\eta_3= \varepsilon_1-\varepsilon_2+2\varepsilon_3.
\end{equation}
According to the usual convention, to an arbitrary subset of the simple positive roots of $so(7)$ (``crossed-out'' roots) we assign a parabolic subalgebra by requesting that the crossed out root spaces lie outside of the Levi part of $\gop$.  In turn, we parametrize the subsets of the simple positive roots of $so(7)$ by triples of $0$'s and $1$'s with $1$ standing for ``crossed-out'' root. Finally, we index the parabolic subalgebra by the corresponding triples of $0$'s and $1$'s. For example, by $\gop_{(1,1,0)}$ we denote the parabolic subalgebra of $so(7)$ whose Levi part has roots $\pm \varepsilon_3$. Define the four parabolic subalgebra $\gob'\simeq \gop'_{(1,1)},\gop'_{(1,0)},\gop'_{(0,1)}, \gop'_{(0,0)} \simeq \LieGtwo$ of $\LieGtwo$ in analogous fashion.

We recall from \cite{ms} that the pairwise inclusions between the 
parabolic subalgebras of $so(7)$ and the embeddings of the parabolic 
subalgebras of $\LieGtwo$ are given as follows.

\newcommand{\BThreeDiagram}[3]{
\gop_{(#1,#2,#3)}
}

\newcommand{\GTwoDiagram}[2]{
{\gop}'_{(#1,#2)}
}
\begin{lemma}\label{leParabolicsG2inParabolicsB3}
For the pair $G_2\stackrel{i}{\hookrightarrow} so(7)$, 
let $\goh, \gob,\gop$ denote Cartan, 
Borel and parabolic subalgebras of $so(7)$ and $\goh', {\gob}', {\gop}'$ denote Cartan, 
Borel and parabolic subalgebras of $\LieGtwo$
with the assumptions that 
$i(\goh')\subset\goh\subset \gob$, $i({\gob}')\subset \gob\subset\gop$, 
${\gob}'\subset{\gop}'$.
Then we have the following inclusion diagram for the possible values of $\gop, \gop'$.
\[
\xymatrix{
&{\BThreeDiagram{0}{0}{0}}\simeq so(7) & &  \\
{\BThreeDiagram{1}{0}{0}} \ar[ur] &{\BThreeDiagram{0}{1}{0}}  \ar[u] &{\BThreeDiagram{0}{0}{1}} \ar[ul] &
{\GTwoDiagram{0}{0}} \simeq\LieGtwo \ar[ull] \\
{\BThreeDiagram{1}{1}{0}} \ar[u]  \ar[ur] &{\BThreeDiagram{1}{0}{1}}  \ar[ul] \ar[ur]  &{\BThreeDiagram{0}{1}{1}} \ar[ul]  \ar[u] &
& {\GTwoDiagram{0}{1}}  \ar[ul] \ar[ulll]  \\
&{\BThreeDiagram{1}{1}{1}}\simeq \gob \ar[u]  \ar[ur]  \ar[ul]  &  &
{\GTwoDiagram{1}{0}} \ar[ull] \ar[uu]  \\
&&& {\GTwoDiagram{1}{1}}\simeq \gob' \ar[u] \ar[ull]  \ar[uur]\\
}
\]
The arrows in the diagram indicate 
the inclusions between the corresponding parabolic subalgebras. 
In addition, if an arrow  is drawn between 
the parabolic subalgebra $\gop'$ of 
$\LieGtwo$ and a parabolic subalgebra $\gop$ of $so(7)$, 
then $\gop'= i^{-1}(i({\gog}')\cap\gop)$.
\end{lemma}
The structure of $so(7)$ as a module over the Levi part of parabolic subalgebras of $\LieGtwo$ is described in detail in \cite[Lemma 5.2]{ms} (the lemma is too large to recall here) and we will implicitly use it throughout Section \ref{secG2SingularVectorsComputedWithInvariants}.

Note that the conformal parabolic subalgebra
$\gop_{(1,0,0)}\subset so(7)$ and the parabolic subalgebra $\gop'_{(1,0)}\subset \LieGtwo$ are not compatible.

%%%%%%%%%%%%%%%%%%%%%%%%%%%%%%%%%%%%%%%%%%%%%%%%%%%%%%%%%%%%%%%%%%%%%%%%%%%%%%%%%%%%%%%%%%%%%%%%

%%%%%%%%%%%%%%%%%%%%%%%%%%%%%%%%%%%%%%%%%%%%%%%%%%%%%%%%%%%%%%%%%%%%%%%%%%%%%%

\section{Branching problem and the F-method (algebraic distribution Fourier transformation)}

In the present section we briefly review the F-method developed in \cite{KOSS1}, \cite{KOSS2}. It is based on the analytical tool 
of algebraic Fourier transformation on the commutative nilradical $\gon$
of $\gop$, which allows to find singular vectors in generalized Verma 
modules exploiting the algebraic Fourier transform and classical invariant theory.
The method converts a problem in the universal enveloping algebra of a Lie algebra into a system 
of partial or ordinary special differential equations acting on a polynomial ring. In examples known to us, the conversion 
to partial differential equations yields a lot more tractable problem than the starting universal enveloping algebra one.

Let $\tilde{G}$ be a connected real reductive Lie group
 with the Lie algebra $\tilde\gog$, $\tilde P\subset \tilde G$ 
 a parabolic subgroup and $\tilde\gop$ its Lie algebra, 
 $\tilde\gop=\tilde\gol\oplus\tilde\gon$
the Levi decomposition of $\tilde\gop$ and $\tilde\gon_-$ its opposite nilradical,
$\tilde\gog=\tilde\gon_-\oplus\tilde\gop$. The corresponding Lie 
groups are denoted $\tilde N_-,\tilde L,\tilde N$. 
Let $p$ denote the fibration $p:\tilde G\to \tilde G/\tilde P$ and let
$\tilde M:=p(\tilde N_-\cdot\tilde P)$ denote the big Schubert cell of $\tilde G/\tilde P$. Then the exponential map
$$
\tilde\gon_-\to \tilde M, 
\quad
X\mapsto\exp(X)\cdot o\in \tilde G/\tilde P,\, o:= e\cdot \tilde P\in \tilde G/\tilde P, \, e\in\tilde G.
$$
 gives the canonical identification of the vector space $\gon_-$ with $\tilde M$.
 
Given a complex finite dimensional $\tilde P$-module $\mV$ (in the present section we do not indicate explicitly its highest weight), let ${Ind}_{\tilde P}^{\tilde G}(\mV)$ denote the space 
of smooth sections of the homogeneous vector bundle
$\tilde G \times_{\tilde P} \mV \to \tilde G/\tilde P$, i.e.,
$$
{Ind}_{\tilde P}^{\tilde G}(\mV)=C^\infty(\tilde G,\mV)^{\tilde P}:=\{f\in C^\infty(\tilde G,\mV)|
f(g\cdot p)=p^{-1}\cdot f(g),\, g\in \tilde G,p\in \tilde P\}.
$$
Let $\tilde\pi$ denote the induced representation of $\tilde G$ on  ${Ind}_{\tilde P}^{\tilde G}(\mV)$.

Let ${\fam2 U}(\tilde\gog_\mC)$ denote the universal enveloping algebra
of the complexified Lie algebra $\tilde{\gog}_\mC$.
Let $\mV^\ch$ be the dual (contragredient) representation to $\mV$.
The generalized Verma module ${M}^{\tilde\gog}_{\tilde\gop}(\mV^\ch)$ is defined by
$$
{M}^{\tilde\gog}_{\tilde\gop}(\mV^\ch):= 
{\fam2 U}(\tilde\gog)\otimes_{{\fam2 U}(\tilde\gop)}\mV^\ch,  
$$
and there is a $(\tilde{\mathfrak {g}}, \tilde P)$-invariant natural
pairing between 
 $\Ind_{\tilde P}^{\tilde G}(\mV)$  and  ${M}^{\tilde\gog}_{\tilde\gop}(\mV^\ch)$,  
described as follows. Let ${\lD}'(\tilde G/\tilde P)\otimes \mV^\ch$ be 
the space 
of all distributions on $\tilde G/\tilde P$ with values in $\mV^\ch$.
The evaluation defines a canonical equivariant pairing between $\Ind^{\tilde G}_{\tilde P}(\mV)$ and
 ${\lD'}(\tilde G/\tilde P)\otimes \mV^\ch$, and this restricts to the pairing
\begin{equation}\label{distrduality}
\Ind^{\tilde G}_{\tilde P}(\mV)\times {\lD}_{[o]}'(\tilde G/\tilde P)\otimes \mV^\ch\to \mC,
\end{equation}
where $ {\lD}'(\tilde G/\tilde P)_{[o]}\otimes \mV^\ch$ denotes the space of distributions
supported at the base point $o\in \tilde G/\tilde P.$
As shown in \cite{cs}, the space 
${\lD}_{[o]}'(\tilde G/\tilde P)\otimes \mV^\ch$ can be identified,
as an ${\fam2 U}(\tilde\gog)$-module, 
with the generalized Verma module ${M}^{\tilde\gog}_{\tilde\gop}(\mV^\ch).$

Moreover, given two inducing representations $\mV$ and $\mV'$ of $\tilde P$, the space of $\tilde G$-equivariant differential operators
from $\Ind_{\tilde P}^{\tilde G}(\mV)$ to  
$\Ind_{\tilde P}^{\tilde G}(\mV')$ is isomorphic to
 the space of $(\tilde\gog,\tilde P)$-homomorphisms
between ${M}^{\tilde\gog}_{\tilde\gop}(\mV'^\ch)$ and ${M}^{\tilde\gog}_{\tilde\gop}(\mV^\ch)$.  
The homomorphisms of generalized Verma modules are determined by their singular vectors,
and the F-method translates the problem of finding singular vectors to the study of distributions on 
$\tilde G/\tilde P$ supported at the origin, and consequently to the problem of finding the solution space
for a system of partial differential equations acting on polynomials $\Pol(\tilde\gon)$ on $\tilde{\mathfrak {n}}$.  

The representation $\tilde\pi$ of $\tilde G$ on $\Ind_{\tilde P}^{\tilde G}(\mV)$ has the infinitesimal 
representation $d\tilde\pi$ of $\tilde\gog_\mC$. In the non-compact case, $\tilde \pi$ acts on functions on the big Schubert cell $\tilde\gon_-\simeq \tilde M\subset \tilde G/\tilde P$
with values in $\mV$. The latter representation space can be identified via the exponential 
map with $C^\infty(\tilde\gon_-,\mV)$. The action $d\tilde\pi(Z)$
of elements $Z\in \tilde\gon$
 on $\lC^\infty(\tilde\gon_-,V)$ is realized by vector fields on 
 $\tilde{\mathfrak {n}}_-$ with coefficients in $\Pol(\tilde\gon_-)\otimes\End \mV$, see \cite{kos}.

By the Poincar\'e-Birkhoff-Witt theorem, 
the generalized Verma module ${M}^{\tilde\gog}_{\tilde\gop}(\mV^\ch)$ is isomorphic
 to ${\fam2 U}(\tilde\gon_-)\otimes \mV^\ch\simeq\Diff_{\tilde N_-} (\tilde\gon_-)\otimes \mV^\ch$ 
 as an $\tilde{\mathfrak {l}}$-module.  
In the special case when $\tilde\gon_-$ is commutative, 
$\Diff_{\tilde N_-} (\tilde\gon_-)$ is the space of holomorphic differential 
operators on $\tilde\gon_-$ with constant coefficients regarded as a 
subspace of the Weyl algebra 
$\Diff (\tilde\gon_-)$ of algebraic differential operators on $\tilde\gon_-$.
Moreover, the operators $d\tilde\pi^{\ch}(X),\,X\in\tilde\gog$,
are realized as differential operators
on $\tilde\gon_-$ with coefficients in $\End(\mV^\ch).$
The application of Fourier transform
on $\tilde\gon_-$ gives the identification of the generalized Verma module  
$\Diff_{\tilde N_-} (\tilde\gon_-)\otimes \mV^\ch$
with the space 
$\Pol(\tilde\gon)\otimes \mV^\ch$,
and the action $d\tilde\pi^{\ch}$ of $\tilde\gog$
on $\Diff_{\tilde N_-} (\tilde\gon_-)\otimes \mV^\ch$
translates to the action $(d\tpi^\ch)^F$ of $\tilde\gog$
on $\Pol(\tilde\gon)\otimes \mV^\ch$ and is realized again by differential operators
with values in $\End(\mV^\ch)$. The explicit form of $(d\tpi^\ch)^F(X)$
is easy to compute by Fourier transform from the explicit form of $d\tilde\pi^{\ch}.$
 
The previous framework can be applied to any pair of couples $\tilde P\subset \tilde G$ and $\tilde P'\subset \tilde G'$ of Lie groups for which $\tilde G'\subset \tilde G$ is a reductive subgroup of $\tilde G$ and 
$\tilde P'=\tilde P \cap \tilde G'$ is a parabolic subgroup of $\tilde G'$. The Lie algebras
of $\tilde G',\tilde P'$ are denoted by $\tilde{\mathfrak {g}}',\tilde{\mathfrak {p}}'$. 
In this case, 
$\tilde {\mathfrak {n}}':=\tilde{\mathfrak {n}} \cap \tilde{\mathfrak {g}}'$
is the nilradical 
of $\tilde{\mathfrak {p}}'$, 
and $\tilde L'=\tilde L \cap \tilde G'$
is the Levi subgroup of $\tilde P'$.  
We are interested in the branching problem for generalized Verma modules 
${M}^{\tilde\gog}_{\tilde\gop}(\mV^\ch)$
over, $\tilde{\mathfrak {g}}$, i.e., in the structure of the restriction of ${M}^{\tilde\gog}_{\tilde\gop}(\mV^\ch)$ to $\tilde{\mathfrak {g}}'$. 

\begin{definition}
Let $\mV$ be an irreducible $\tilde P$-module.
Define the $\tilde L'$-module
\begin{eqnarray}
{M}_{\tilde \gop}^{\tilde \gog}(\mV^\ch)^{\tilde\gon'}
 :=
 \{v\in {M}^{\tilde\gog}_{\tilde\gop}(\mV^\ch)|\, d\pi^\ch (Z)v=0
\textrm{  for all } Z\in \tilde\gon'\}.
\end{eqnarray}
\end{definition}
The set 
${M}_{\tilde\gop}^{\tilde\gog}(\mV^\ch)^{\tilde\gon'}$ is a completely reducible  $\tilde{\mathfrak {l}}'$-module. Note that for $\tilde G=\tilde G'$, ${M}_{\tilde\gop}^{\tilde\gog} (\mV^\ch)^{ \tilde\gon'}$ is necessarily finite-dimensional. However for $\tilde G\not = \tilde G'$, the set ${M}_{\tilde \gop}^{\tilde\gog}(V^\ch)^{\gon'}$ will in general (but not necessarily, as illustrated in the next section) be infinite dimensional. An irreducible $\tilde L'$-submodule $\mW^{\ch}$ of 
${M}_{\tilde\gop}^{\tilde\gog}(\mV^\ch)^{\tilde\gon'}$ gives
an injective ${\fam2 U}(\tilde\gog')$-homomorphism from ${M}_{\tilde\gop'}^{\tilde\gog'}(\mW^\ch)$
to ${M}_{\tilde\gop}^{\tilde\gog}(\mV^\ch)$. Dually, 
we get an equivariant differential operator acting from
$\Ind_{\tilde P}^{\tilde G}(\mV)$ to $\Ind_{\tilde P'}^{\tilde G'}(\mW)$.  

Using the F-method, the space of $\tilde L'$-singular vectors 
${M}_{\tilde\gop}^{\tilde\gog}(\mV^\ch)^{\tilde\gon'}$ is 
realized in the ring of polynomials on 
$\tilde\gon$ valued in $\mV^\ch$ and equipped with the action of the Lie algebra via $(d\tpi^\ch)^F.$

\begin{definition}  
We define
\begin{eqnarray}
\label{eqn:sol}
Sol(\tilde{\gog},\tilde{\gog}',\mV^{\ch}) 
:=\{f \in \Pol(\tilde{\gon})\otimes \mV^\ch |\,
(d\tpi^\ch)^F(Z) f = 0 \textrm{ for all } Z\in\tilde{\gon}' 
\}.
\end{eqnarray}
\end{definition}

Then the inverse Fourier transform gives an $\tilde L'$-isomorphism
\begin{equation}
\label{eqn:phi}
Sol(\tilde{\gog},\tilde{\gog}';\mV^\ch) \overset \sim \to
{M}_{\tilde{\gop}}^{\tilde{\gog}}(\mV^{\ch})^{\tilde{\gon}'}.  
\end{equation}
An explicit form
of the action $(d\tpi^\ch)^F(Z)$ leads to 
a system  of differential equation for elements in $ \mbox{\rm Sol}$.
The transition from ${M}_{\tilde\gop}^{\tilde\gog}(\mV^\ch)^{\tilde\gon'}$ to $ \mbox{\rm Sol}$ 
transforms the problem of computation of singular 
vectors in generalized Verma modules into a system of partial differential equations.

In the dual language of differential operators acting on 
principal series representation, the set of $\tilde G'$-intertwining differential operators from
$\Ind_{\tilde P}^{\tilde G}(\mV)$ to $\Ind_{\tilde P'}^{\tilde G'}(\mV')$ is in bijective
correspondence with the space of all $(\tilde\gog',\tilde P')$-homomorphisms
from  ${M}^{\tilde\gog'}_{\tilde\gop'}({\mV'}^\ch)$ to ${M}^{\tilde\gog}_{\tilde\gop}(\mV^\ch).$

%%%%%%%%%%%%%%%%%%%%%%%%%%%%%%%%%%%%%%%%%%%%%%%%%%%%%%%%%%%%%%%%%%%%%%%%%%%%%%%
%%%%%%%%%%%%%%%%%%%%%%%%%%%%%%%%%%%%%%%%%%%%%%%%%%%%%%%%%%%%%%%%%%%%%%%%%%%%%%%%

\section{$\LieGtwo\cap\gop' $-singular vectors in the $so(7)$-generalized Verma modules of scalar type for the conformal parabolic subalgebra}\label{secG2SingularVectorsComputedWithInvariants}

In this subsection we determine the $i(\LieGtwo)\cap \gop$-singular vectors 
in the family of $\tilde\gog=so(7)$ generalized Verma modules $M^{so(7)}_{\gop_{(1,0,0)}}(\mC_\lambda)$ induced 
from character $\chi_\lambda :\tilde\gop\to\mC$ of the weight $\lambda\varepsilon_1$ ($\varepsilon_1$ is the 
first fundamental weight of 
$so(7)$). In this way, the results computed in the present section are analytic 
counterpart realized by F-method of the algebraic results developed in \cite{ms}.

Denote by $v_\lambda$ the highest weight vector of the generalized Verma $so(7)$-module 
$M^{so(7)}_{\gop_{(1,0,0)}}(\mV_\lambda)$. 
Note that as  $i(h_2')=3h_2=3h_{\varepsilon_2-\varepsilon_3}$, 
$i(h_1')=h_1+2h_2=h_{\varepsilon_1-\varepsilon_2} + 2h_{\varepsilon_3}$, $\langle\mu,\alpha_1\rangle =0$ and $\langle\mu,\alpha_2\rangle =\lambda$, we have that
the $\goh'$-weight of $v_\lambda$ is $\mu=\lambda(\alpha_1+2\alpha_2)$.

Let $\gon_-$ denote the nilradical opposite to the nilradical of the parabolic subalgebra $\gop$. Then $\gon_-$ is commutative, 
$$
{\lU}(\gon_-)\otimes \mV^\ch\simeq \Pol\left(\frac{\pa}{\pa x_1},
\dots, \frac{\pa}{\pa x_5}\right)\otimes \mC_\lambda %\simeq \Pol\left(\frac{\pa}{\pa x_1},\dots, \frac{\pa}{\pa x_5}\right)
$$
and the variables $\frac{\pa}{\pa x_1},\dots ,\frac{\pa}{\pa x_5}$
denote the following $so(7)$-root space generators: 
\[
\begin{array}{lll}
\frac{\pa}{\pa x_1} :=g_{-\varepsilon_1+\varepsilon_2}=g_{-1}, &
\frac{\pa}{\pa x_2}:=g_{-\varepsilon_1-\varepsilon_3}= g_{-8}, & 
\frac{\pa}{\pa x_3}:=g_{-\varepsilon_1} = g_{-6} ,
\\
\frac{\pa}{\pa x_5}:=g_{-\varepsilon_1+\varepsilon_3}= g_{-4}, &  \frac{\pa}{\pa x_4}:=g_{-\varepsilon_1-\varepsilon_2}=g_{-9}.
\end{array}
\]
Here, we recall that 
$
[x_i, \frac{\partial }{\partial x_j}]= -[\frac{\partial }{\partial x_j}, x_i]=
\left\{
\begin{array}{cc}
0 &\mathrm{if~}i\neq j\\
-1 &\mathrm{if~}i=j 
\end{array}\right. 
$ is the adjoint action of the differential operator $x_i$ on the differential operator $\frac{\partial }{\partial x_j} $.

By Lemma \ref{leParabolicsG2inParabolicsB3}, the simple part of the Levi factor of $i(\gop')$ is isomorphic to $sl(2)$ and its action on $\gon_-$ can be extended to action on ${\lU}(\gon_-)\simeq S^\star (\gon_-)$. The elements $h:=h_2, e:=g_2, f:=g_{-2}$ give the standard $h,e,f$-basis of $sl(2)$, i.e., $[e,f]=h, [h,e]=2e, [h, f]=-2f$.  Then the action of $h$ on $\gon_-$ is the adjoint action of $ x_{5} \px{5} +x_{4} \px{4} -x_{2} \px{2} -x_{1} \px{1} $, the action of $e$ is the adjoint action of $x_{4} \px{2}-x_{5} \px{1}$
and the action of $f$ is  the adjoint action of $ -x_{1} \px{5}+x_{2} \px{4}$.

We now proceed to generate all $\gol'$-invariant singular vectors in $M^{so(7)}_{\gop_{(1,0,0)}}(\mC_\lambda)$, i.e., the singular vectors that induce $i(\LieGtwo)$-generalized Verma modules induced from character (scalar generalized Verma modules). To do that we need the following lemma from classical invariant theory of reductive Lie algebras. 

\begin{lemma}\label{invariants}
Then the $sl(2)$-invariants of $S^\star(\gon_-)$ are an associative algebra generated by the elements 
\begin{equation}\label{eqLeInvariants}
\begin{array}{rcl}
\displaystyle u_1&\displaystyle :=&\displaystyle \px{1}\px{4}+\px2\px5 =g_{-1}g_{-9}+g_{-8}g_{-4}  \\
\displaystyle u_2&\displaystyle :=&\displaystyle \px3= g_{-6}\quad .
\end{array}
\end{equation}
\end{lemma}
\begin{proof}
Direct computation shows that $u_1$, $u_2$ are invariants. Alternatively, as the direct sum of two two-dimensional $sl(2)$-modules gives a natural embedding 
$sl(2)\hookrightarrow sl(2)\times sl(2) $, 
we can view $u_1$ as the invariant element induced by the defining 
symmetric bilinear form of $so(4)\simeq sl(2)\times sl(2)$. 
Let the positive root of $sl(2)$ be\footnote{$\eta$ is of course the projection of long $\LieGtwo$-root $ \alpha_2=\pr (\varepsilon_2-\varepsilon_3) $ from the dual of the two-dimensional Cartan subalgebra of $\LieGtwo$ to the dual of the one-dimensional Cartan subalgebra of a long-root $sl(2)$-subalgebra of $\LieGtwo$} $\eta$ , and the multiplicity of the $sl(2)$-module 
with highest weight $t\frac\eta 2$ in $S^l(\gon_-)$ be $b(l,t)$. Denoting by $x,z$ a couple 
of formal variables, we have that 
$\sum_{l\in \mZ_{\geq 0}, t\in \mZ_{\geq 0}} b(l,t)(z^l x^{t}+z^lx^{-1-t})$ is the power 
series expansion of the rational function
\begin{eqnarray*}
(1-x^{-2}) \frac{1}{(1-zx)^2} \frac{1}{(1-zx^{-1})^2}\frac{1}{(1-z)}.
\end{eqnarray*}
Direct computation shows that $b(l,t)$ equals $-1/2t^{2}+1+1/2tl+1/2l+1/2t$  whenever $l+t$ is even and  $ -1/2t^{2}+1/2+1/2tl+1/2l$
whenever $l+t$ is odd, and $l$, $t$ satisfy the inequalities $l\geq t \geq  0$. Finally, 
substituting with $t=0$, we get $b(l,0)=1+l/2$ for even $l$ and $b(l,0)=1/2+l/2$. For a fixed $l$, 
this is exactly the dimension of the vector space generated by the linearly independent invariants 
$u_1^qu_2^r\in S^l(\gon_-)$ with $r+2q=l$, 
which completes the proof of our Lemma.
\end{proof}

%We note that in order to extend the technique of this paper to
%non-scalar Verma modules, one would need non-classical versions of
%the above Lemma using differential operators with coefficients in
%the 
% endomorphism ring of finite dimensional
%representations of $\gol'$.  

From the definition of embedding map $i$ it follows that 
\[
\begin{array}{rcl}
\displaystyle\ad(i(g_2')) &=&\displaystyle -x_2\px4 +x_1\px5,\\
\displaystyle\ad(i(g_{-2}')) &=&\displaystyle -x_4\px2 +x_5\px1,\\
\displaystyle\frac{1}3\ad(i(h_2'))&=& \displaystyle\ad(h_2) = [\ad(i(g_2')), \ad(i(g_{-2}'))]\\&=&
\displaystyle x_{5} \px{5} +x_{4} \px{4} -x_{2} \px{2} -x_{1} \px{1}, \\
\displaystyle \ad (i(h_1')) &=&\displaystyle -x_{5} \px{5} +x_{3} \px{3} +3x_{2} \px{2} +2x_{1} \px{1},
\end{array}
\]
and therefore 
\begin{equation}\label{eqGradingElement}
\ad(i(2h_1'+h_2'))=x_{5} \px{5} +3x_{4} \px{4} +2x_{3} \px{3} +3x_{2} \px{2} +x_{1} \px{1}
\end{equation}
represents the central element of 
the Levi factor $i(\gol')$. Its action therefore naturally induces a grading $\gr$ on the Weyl algebra of $\gon_-$ in the variables 
$$
\left\{x_1,x_2,x_3,x_4,x_5,\frac{\pa}{\pa x_1},\frac{\pa}{\pa x_2},\frac{\pa}{\pa x_3},
\frac{\pa}{\pa x_4},\frac{\pa}{\pa x_5}\right\},
$$
via 
\begin{equation}\label{grading}
\begin{array}{ll}
 -\gr\left(x_1\right)=\gr\left(\px1\right)=-1,\, & -\gr\left(x_2\right)=\gr\left(\px2\right)=-3,\\
-\gr(x_3)=\gr\left(\px3\right)=-2,\, & -\gr(x_4)=\gr\left(\px4\right)=-3,\\
-\gr\left(x_5\right)=\gr\left(\px5\right)=-1\quad .
\end{array}
\end{equation}
In particular, we get that the invariants $u_1=\frac{\pa}{\pa x_1}\frac{\pa}{\pa x_4}+\frac{\pa}{\pa x_2}\frac{\pa}{\pa x_5}$ and 
$u_2=\left(\frac{\pa}{\pa x_3}\right)^2$ are homogeneous with respect to the $\gr$-grading.

Let  $\xi_1, \dots \xi_5$ be formal variables, Fourier-dual with respect to $x_1, \dots, x_n$. Let 
$$ 
\pa_1:=\frac{\pa}{\pa \xi_1},\dots, \pa_5:=\frac{\pa}{\pa \xi_5},
$$
denote the derivatives in the $\xi_i$-variables. We recall that the distributive Fourier transform ${\fam2 F}$ maps the Weyl algebra generated by $x_1,\dots, x_n,\px1,\dots, \px5$ to the Weyl algebra generated by $\pa_1,\dots, \pa_5,\xi_1,\dots, \xi_5$ via 
\[
{\fam2 F}(x_i):=\pa_i\quad \quad {\fam2 F}\left(\px{i}\right):=\xi_i\quad .
\]
As the Fourier transform is a Lie algebra homomorphism,
by Lemma \ref{invariants} the subalgebra
of $\gol'_s=sl(2)$-invariants with respect to the Fourier dual representation
is the polynomial ring $Pol[\xi_1\xi_4+\xi_2\xi_5,\xi_3]$. 
\begin{theorem}\label{thMain}
Let $v_\lambda$ be the highest weight vector of the $so(7)$-generalized Verma module 
$M^{so(7)}_{\gop_{(1,0,0)}}(\mC_\lambda)$ induced from character $\chi_\lambda$, $\lambda\in\mC$. 
Let $N\in\mN$ be a positive integer and $A_i\in\mC$, $i\in\mN$ a collection of
complex numbers such that at least one of them is non-zero. Let 
\begin{equation}\label{eqTheSingularVector}
u\cdot v_\lambda:= \sum_{k=0}^N A_k u_1^{k}u_2^{N-k} \cdot v_\lambda \quad ,
\end{equation}
where $u_1, u_2$ are given by \eqref{eqLeInvariants}.
\begin{enumerate}
\item
A vector $u\cdot v_\lambda$ is 
$i(\LieGtwo)\cap \gop$-singular (``singular vector of scalar type'') of homogeneity $2N$ if and only if $\lambda=N-5/2$ and $u=\left( 2u_1 +u_2 \right)^N=\left( 2u_1 +u_2 \right)^{\lambda+5/2}$.
\item
$M^{so(7)}_{\gop_{(1,0,0)}}(\mC_\lambda)$ has no  $i(\LieGtwo)\cap \gop$-singular vector of homogeneity $2N+1$.
\item A vector $v\in M^{so(7)}_{\gop_{(1,0,0)}} (\mC_\lambda)$, not proportional to $v_\lambda$, is $so(7)\cap\gop$-singular if and only if $\lambda=N-5/2$ and $v=u\cdot v_\lambda$ is the vector given in 1.
\end{enumerate}
\end{theorem}
\noindent
\begin{proof}
1. By Lemma \ref{invariants} and Section \ref{secso7G2LieAlgstructure} a $\gop'$-singular vector must be polynomial in $u_1$ and $u_2$ and therefore a homogeneous $\gop'$-singular vector of homogeneity $2N$ must be of the form \eqref{eqTheSingularVector}.

First  we determine the action of the second simple positive root $g_2$
in the Fourier dual representation $d\tilde{\pi}(\ad(i(g_1')))$, acting on $Pol[\xi_1,\dots ,\xi_5]$.

Let $n_i$ be non-negative integers. Then 
\begin{eqnarray}
& & i(g_1')\cdot ( \xi_1^{n_1}\xi_2^{n_2}\xi_3^{n_3}\xi_4^{n_4}\xi_5^{n_5} \cdot v_\lambda)= 
\nonumber \\ \nonumber
& &
\left( (-n_1^{2}+n_1)\xi_1^{n_1-1}\xi_2^{n_2}\xi_3^{n_3}\xi_4^{n_4}\xi_5^{n_5}
-n_2 \xi_1^{n_1}\xi_2^{n_2-1}\xi_3^{n_3+1}\xi_4^{n_4}\xi_5^{n_5}
\right.
\\ \nonumber
& & +n_1 \lambda\xi_1^{n_1-1}\xi_2^{n_2}\xi_3^{n_3}\xi_4^{n_4}\xi_5^{n_5}
+(n_3^{2}-n_3) \xi_1^{n_1}\xi_2^{n_2}\xi_3^{n_3-2}\xi_4^{n_4+1}\xi_5^{n_5}    
\\ \nonumber
& & +2n_3   \xi_1^{n_1}\xi_2^{n_2}\xi_3^{n_3-1}\xi_4^{n_4}\xi_5^{n_5+1}    
-n_1n_5  \xi_1^{n_1-1}\xi_2^{n_2}\xi_3^{n_3}\xi_4^{n_4}\xi_5^{n_5}           
\\ \nonumber 
& &
+n_2n_5  \xi_1^{n_1}\xi_2^{n_2-1}\xi_3^{n_3}\xi_4^{n_4+1}\xi_5^{n_5-1} 
-n_1n_2  \xi_1^{n_1-1}\xi_2^{n_2}\xi_3^{n_3}\xi_4^{n_4}\xi_5^{n_5}    
\\ \nonumber 
& &
\left. -n_1n_3  \xi_1^{n_1-1}\xi_2^{n_2}\xi_3^{n_3}\xi_4^{n_4}\xi_5^{n_5}   
\right)\cdot v_\lambda
\\ \nonumber 
& & = 
(-\xi_1\partial_1^2 -\xi_3\partial_2+\lambda\partial_1+\xi_4\partial_3^2  
+2\xi_5\partial_3 -\xi_5\partial_1\partial_5+\xi_4\partial_2\partial_5 
\\
& & -\xi_2\partial_1\partial_2 -\xi_3\partial_1\partial_3)
\cdot ( \xi_1^{n_1}\xi_2^{n_2}\xi_3^{n_3}\xi_4^{n_4}\xi_5^{n_5} )\cdot v_\lambda,
\end{eqnarray}
Let $P(\lambda)$ denote the differential operator on 
$\mathbb C[\xi_1,\xi_2, \xi_3, \xi_4, \xi_5]$ obtained 
in the following computation:
\begin{equation}\label{eqad_i(g_2)}
\begin{array}{rcl}
& &(-\xi_1\partial_1^2 -\xi_3\partial_2+\lambda\partial_1+\xi_4\partial_3^2  +2\xi_5\partial_3\\
&&
-\xi_5\partial_1\partial_5+\xi_4\partial_2\partial_5 
\nonumber  -\xi_2\partial_1\partial_2 -\xi_3\partial_1\partial_3)\\
&=&
(-\xi_3\partial_2+\xi_4\partial_3^2  +2\xi_5\partial_3 +(-\xi_5\partial_1+\xi_4\partial_2)\partial_5 
\\&&
\nonumber  -(\xi_1\partial_1+\xi_2\partial_2 +\xi_3\partial_3-\lambda)\partial_1)\\
&=&
(-\xi_3\partial_2+\xi_4\partial_3^2  +2\xi_5\partial_3 
\nonumber +\partial_5 (-\xi_5\partial_1+\xi_4\partial_2) \\&&
-(\xi_1\partial_1+\xi_2\partial_2 +\xi_3\partial_3-\lambda -1)\partial_1).
\end{array}
\end{equation}
We compute
\begin{eqnarray*}
\partial_1\cdot (u_1^{b_1}u_2^{b_2})&=&b_1\xi_4 u_{1}^{b_1-1} u_2^{b_2}, \\
\partial_2\cdot (u_1^{b_1}u_2^{b_2})&=&b_1\xi_5 u_{1}^{b_1-1} u_2^{b_2}, \\
(\xi_1\partial_1+\xi_2\partial_2 )\cdot (u_1^{b_1}u_2^{b_2})&=&b_1u_{1}^{b_1} u_2^{b_2}, \\
\partial_3\cdot (u_1^{b_1}u_2^{b_2})&=&2b_2\xi_3u_{1}^{b_1} u_2^{b_2-1}, \\
\partial_3^2\cdot (u_1^{b_1}u_2^{b_2})&=&2b_2(2b_2-1) u_{1}^{b_1} u_2^{b_2-1},
\end{eqnarray*}
and so
\begin{eqnarray}\notag
&&(-\xi_3\partial_2+\xi_4\partial_3^2  +2\xi_5\partial_3 +\partial_5 (-\xi_5\partial_1+\xi_4\partial_2)- \notag\\
& & (\xi_1\partial_1+\xi_2\partial_2 +\xi_3\partial_3-\lambda-1)\partial_1) \cdot (u_1^{b_1}u_2^{b_2})\notag\\ 
&=&\notag
(-\xi_3\partial_2+\xi_4\partial_3^2  +2\xi_5\partial_3  
-(\xi_1\partial_1+\xi_2\partial_2 +\xi_3\partial_3-\lambda-1)\partial_1) \cdot (u_1^{b_1}u_2^{b_2})\\&=&\notag
 -b_1\xi_3\xi_5 u_{1}^{b_1-1} u_2^{b_2}+2b_2(2b_2-1)\xi_4 u_{1}^{b_1} u_2^{b_2-1} 
 +4b_2\xi_5 \xi_3u_{1}^{b_1} u_2^{b_2-1} \\&&\notag
 -(\xi_1\partial_1+\xi_2\partial_2 +\xi_3\partial_3-\lambda-1)\cdot( b_1\xi_4 u_{1}^{b_1-1} u_2^{b_2})\\&=&\notag
-b_1\xi_3\xi_5 u_{1}^{b_1-1} u_2^{b_2}+2b_2(2b_2-1)\xi_4 u_{1}^{b_1} u_2^{b_2-1} +4b_2\xi_5 \xi_3u_{1}^{b_1} u_2^{b_2-1}  \\
&&\notag
+ (-b_1+1+\lambda +1-2b_2) b_1\xi_4 u_{1}^{b_1-1} u_2^{b_2} \\&=&\label{eqxiActsOnu_1u_2}
2b_2 (  (2b_2-1)\xi_4 +2\xi_5\xi_3)u_{1}^{b_1} u_2^{b_2-1} \notag\\
& & + b_1((-b_1-2b_2+\lambda +2) \xi_4 -\xi_3\xi_5)u_{1}^{b_1-1} u_2^{b_2} \quad .
\end{eqnarray} 
The operator $P(\lambda)$ is homogeneous with respect to the grading in (\ref{grading}), 
and its application to a homogeneous polynomial in $u_1= u_1(\xi_1,\dots ,\xi_5)$, 
$u_2= u_2(\xi_1,\dots ,\xi_5)$ yields
\begin{equation}\label{eqXiActsOnInvariantSumEven}
\begin{array}{rcl}
& & P(\lambda)( \sum_{k=0}^N A_k u_1^{k}u_2^{N-k} )
\nonumber \\
&=& \sum_{k=0}^N A_k( 2(N-k) (  (2(N-k)-1)\xi_4 +2\xi_5\xi_3)u_{1}^{k} u_2^{N-k-1} 
\\&&
+ k((-k-2(N-k)+\lambda +2) \xi_4 -\xi_3\xi_5)u_{1}^{k-1} u_2^{N-k} )\\
&=& \sum_{s=1}^{N+1} 2 A_{s-1} (N-(s-1)) (  
(2(N-(s-1))-1)\xi_4 
\\
& & +2\xi_5\xi_3)u_{1}^{(s-1)} u_2^{N-(s-1)-1} 
\\&&
+\sum_{k=0}^N  kA_k((-k-2(N-k)+\lambda+2) \xi_4 -\xi_3\xi_5)u_{1}^{k-1} u_2^{N-k}\\
&=& \sum_{s=1}^{N} 
(2A_{s-1}(N-s+1)
( (2N-2s+1)\xi_4 +2\xi_5\xi_3)
\\&&
\phantom{\sum_{s=1}^{N} ( }
+sA_s((s-2N+\lambda+2) \xi_4 -\xi_3\xi_5))u_{1}^{s-1} u_2^{N-s}
).
\end{array}
\end{equation}
The $2N$ summands of the form $\xi_4u_{1}^{s-1}u_2^{N-s}$ and $\xi_3\xi_5u_{1}^{s-1}u_2^{N-s}$ 
are linearly independent and therefore the above sum is zero if and only if 
\begin{eqnarray}\label{eqCoeffDensityBranching}
& & 2A_{s-1}(N-s+1)( (2N-2s+1)\xi_4 +2\xi_5\xi_3)
\nonumber \\
& & +sA_s((s-2N+\lambda+2) \xi_4 -\xi_3\xi_5))
\end{eqnarray}
equals zero for all values of $s$. When $s=N$, the above sum becomes 
\[
2A_{N-1}(\xi_4+2\xi_3\xi_5) + NA_N((-N+\lambda +2)\xi_4-\xi_3\xi_5)\quad. 
\]
It is a straightforward check that if $A_N$ vanishes, then $A_{N-1}, A_{N-2},\dots$ must also vanish; therefore we may assume $A_N\neq 0$.
The vanishing of the coefficient in front of $\xi_4$ implies $A_{N-1}=-\frac 12 N A_N\left(- N+\lambda+2 \right)$ and in turn, the vanishing of the coefficient in front of $\xi_3\xi_5$ implies $-5 +2 N -2 \lambda=0$. Therefore
\[
\lambda=N-5/2\quad .
\]
Substituting $\lambda$ back into (\ref{eqCoeffDensityBranching}), we get 
\begin{eqnarray*}
& & 2A_{s-1}(N-s+1)( (2N-2s+1)\xi_4 +2\xi_5\xi_3)
\nonumber \\
& & +sA_s((-N+s-1/2) \xi_4 -\xi_3\xi_5)=0.
\end{eqnarray*}
This implies $A_s=\frac{4(N-s+1)}{s}A_{s-1}=\dots ={4^s}{\binom{N}{s}} A_0$, which completes the proof of 1). 

2. 
A homogeneous $i(\LieGtwo)\cap \gop$-singular vector is, in particular, $sl(2)\simeq i([\gol',\gol'])$-singular and by Lemma \ref{invariants} must be of the form $u= \xi_3\sum_{k=0}^N A_k u_1^{k}u_2^{N-k}$. The application 
of $2\xi_5\partial_3$ converts $A_N(\xi_1\xi_4+\xi_2\xi_5)^N\xi_3$
into $2A_N(\xi_1\xi_4+\xi_2\xi_5)^N\xi_5$. Furthermore $2A_N(\xi_1\xi_4+\xi_2\xi_5)^N\xi_5$ contains in 
its binomial expansion $2A_N(\xi_1\xi_4)^N\xi_5$. Direct check shows that the action of $P(\lambda)$ on 
$(\xi_1\xi_4+\xi_2\xi_5)^{N-i}\xi_3^{1+2i}$ for $i>0$ does not
contain the monomial $(\xi_1\xi_4)^N\xi_5$. This implies that
$A_N=0$ and by induction, the polynomial is trivial. Consequently,
there is no nontrivial odd homogeneity polynomial solving the
differential equation $P(\lambda)$.

As an illustration, for $N=0$ we have $P(\lambda)(A_0\xi_3)=2A_0\xi_5$.
This vanishes provided $A_0=0$, which implies the polynomial
is trivial.  

3. An $so(7)\cap \gop$-singular vector must be $i(\LieGtwo)\cap \gop$-singular. As the grading element from \eqref{eqGradingElement} maps $i(\LieGtwo)\cap \gop$-singular to $i(\LieGtwo)\cap \gop$-singular vectors, it quickly follows that an $i(\LieGtwo)\cap \gop$-singular vector is a linear combination of $\gr$-homogeneous elements (see \eqref{grading}). From the explicit form of $u_1$ and $u_2$ it immediately follows that a homogeneous 
$i(\LieGtwo)\cap \gop$-singular vector is of the form \eqref{eqTheSingularVector}.

From 1) we know that, other than $v_\lambda$, there is at most one more homogeneous  $i(\LieGtwo)\cap \gop$-singular vector, and thus the vector \eqref{eqTheSingularVector} is the only candidate for a $so(7)\cap\gop$-singular vector.  The simple part of $\gol$ is isomorphic to $so(5)$ and induces the quadratic form  with matrix in the coordinates $\xi_1,\dots ,\xi_5$ \begin{eqnarray}
&& Q=
\left(
\begin{array}{ccccc}
0 & 0 & 0 & 2 & 0 \\
0 & 0 & 0 & 0 & 2 \\ 
0 & 0 & 1 & 0 & 0 \\
2 & 0 & 0 & 0 & 0 \\
0 & 2 & 0 & 0 & 0
\end{array}
\right)\,  ,
\nonumber
\end{eqnarray}
i.e., the metric of the form 
$$
g(\xi_1,\xi_2,\xi_3,\xi_4,\xi_5)=(d\xi_3)^2+2(d\xi_1\otimes d\xi_4+d\xi_4\otimes d\xi_1)
+2(d\xi_2\otimes d\xi_5+d\xi_5\otimes d\xi_2).
$$
The Fourier transform of the $so(5)$-invariant Laplace operator associated to $Q$ is  
$$
{\fam2 F}(\triangle_\xi)=Q(\xi_1,\xi_2,\xi_3,\xi_4,\xi_5)=4(\xi_1\xi_4+\xi_2\xi_5)+\xi^2_3.
$$
Relying on $\triangle_\xi$ and the binomial formula for $(4(\xi_1\xi_4+\xi_2\xi_5)+\xi_3^2)^s$, we see 
that the $\LieGtwo\cap\gop$-singular vector constructed 1) is indeed 
$so(7)\cap \gop$-singular. The proof is complete.
\end{proof}

\noindent \textbf{Remark. } As noted in the proof of 3) every $i(\LieGtwo)\cap \gop$-singular is a linear combination of homogeneous $i(\LieGtwo)\cap \gop$-singular vectors, and therefore Theorem \ref{thMain}, 1) and 2) give all $i(\LieGtwo)\cap \gop$-singular vectors (namely, the linear combinations of $v_\lambda$ and the vector given by \eqref{eqTheSingularVector}).

We note that an alternative proof of Theorem \ref{thMain}, 3) can be given as follows. From a well known example (see e.g., \cite{EG}, \cite{KOSS1}, \cite{KOSS2})  of singular vectors in conformal geometry of dimension $5$ describing
conformally invariant powers of the Laplace operator, we know that for $\lambda \in \{-3/2, -1/2, 1/2,\dots \}$ there exists one $so(7)\cap \gop$-singular vector in $M^{so(7)}_{\gop_{(1,0,0)}}(\mC_\lambda)$. On the other hand points 1) and 2) of Theorem \ref{thMain} present us with only one such candidate, so that candidate must be the $so(7)\cap \gop$-singular vector in question.

For $\lambda \in \{-3/2$, $ -1/2, 1/2,\dots \}$, the $\goh$-weight of the  $so(7)\cap\gop$-singular vector in $M^{so(7)}_{\gop_{(1,0,0)}}( \mC_\lambda)$ given by Theorem \ref{thMain} equals $(\lambda- 2N ) \varepsilon_1=(\lambda-2(\lambda+5/2) ) \varepsilon_1=(-\lambda-5)\varepsilon_1$. Therefore the vector from Theorem \ref{thMain} corresponds to the homomorphism of generalized Verma modules 
\begin{equation}\label{eqHmmso(7)}
M^{so(7)}_{\gop_{(1,0,0)}}( \mC_{-\lambda-5}) \hookrightarrow  M^{so(7)}_{\gop_{(1,0,0)}}( \mC_\lambda) \quad .
\end{equation}
In an analogous fashion we conclude that Theorem \ref{thMain} gives a homomorphism of generalized Verma modules
\begin{equation}\label{eqHmmG2}
 M_{\gop'_{(1,0)}}^{\LieGtwo}(\mathbb C_{(-\lambda-5)\psi_1})  \hookrightarrow M_{\gop'_{(1,0)}}^{\LieGtwo}(\mathbb C_{\lambda \psi_1} ) \quad .
\end{equation} 

We note that the existence of the above homomorphisms was proved in \cite{Mat}. We conclude this article by the following.

\begin{proposition}\label{propHMMstandard}
The homomorphisms \eqref{eqHmmso(7)}, \eqref{eqHmmG2} are standard.
\end{proposition}
\begin{proof}
\cite[Chapter 7]{Dixmier} implies that a (non-generalized) Verma module $M_\gob^{\gog}(\mathbb C_\mu)$ lies in a (non-generalized) Verma module $M_\gob^{\gog}(\mathbb C_\nu)$ if and only if there exists a sequence of roots $\alpha_1,\dots, \alpha_k$ such that $s_{\alpha_k}\dots s_{\alpha_1} (\mu+\rho)-\rho =\nu$ and $s_{\alpha_{j+1}} \dots s_{\alpha_{1 }} (\mu+ \rho)- s_{\alpha_j} \dots  s_{\alpha_1}(\mu+ \rho) $ is a positive integer multiple of $\alpha_{j+1}$ for all $j$. Here, $s_{\alpha_i}$ denotes reflection in the root $\alpha_i$ and $\rho$ is the half-sum of the positive roots.

Computation using the above criterion shows that, for $\lambda=-3/2, -1/2, 1/2, \dots$, we have that
\begin{equation}\label{eqStandardHMM1}
M_{\gob}^{so(7)}(\mathbb C_{-\lambda-5})\subset M_{\gob}^{so(7)}(\mathbb C_{\lambda}) 
\end{equation}
and
\begin{equation}\label{eqStandardHMM2}
M_{\gob'}^{\LieGtwo}(\mathbb C_{(-\lambda-5)\psi_1})\subset M_{\gob'}^{\LieGtwo}(\mathbb C_{\lambda\psi_1}) \quad .
\end{equation}
Computation furthermore shows that 
\begin{equation}\label{eqStandardHMM3}
M_{\gob}^{so(7)}(\mathbb C_{-\lambda-5})\nsubseteq M_{\gob}^{so(7)}(\mathbb C_{\mu}) 
\end{equation}
for any $\mu\neq \lambda$ of the form $w(\lambda+\rho)-\rho$, where $w$ is in the Weyl group of $so(7)$ and $\rho$ is the half-sum of the positive roots of $so(7)$. Similarly,
\begin{equation}\label{eqStandardHMM4}
M_{\gob'}^{\LieGtwo}(\mathbb C_{-\lambda-5})\nsubseteq M_{\gob'}^{\LieGtwo}(\mathbb C_{\mu}) 
\end{equation}
for any $\mu\neq \lambda\psi_1$ of the form $w(\lambda\psi_1+ \rho') - \rho'$, where $w$ is in the Weyl group of $\LieGtwo$ and $\rho'$ is the half-sum of the positive roots of $\LieGtwo$.

\eqref{eqStandardHMM1}, \eqref{eqStandardHMM2}, \eqref{eqStandardHMM3}, \eqref{eqStandardHMM4}, together with \cite[Proposition 3.3]{Lepowski:GeneralizationBGG} now imply that the standard homomorphism maps from $M_{\gop(1,0,0)}^{so(7)}(\mathbb C_{-\lambda-5})$ to $M_{\gop(1,0,0)}^{so(7)}(\mathbb C_{\lambda})$ and from $M_{\gop'(1,0)}^{\LieGtwo}(\mathbb C_{(-\lambda-5)\psi_1})$ to $M_{\gop'(1,0)}^{\LieGtwo}(\mathbb C_{\lambda\psi_1})$ are non-zero. On the other hand, our main Theorem \ref{thMain} shows that there is a unique $\gob'$-singular vector of weight $(-5-\lambda)\psi_1 $ in  $M_{\gop'(1,0)}^{\LieGtwo}(\mathbb C_{\lambda\psi_1})$ and therefore also a unique $\gob$-singular vector of weight $(-5-\lambda) \varepsilon_1$ in $M_{\gop(1,0,0)}^{so(7)}(\mathbb C_{\lambda})$. Therefore the homomorphisms \eqref{eqHmmG2}, \eqref{eqHmmso(7)} are standard.
\end{proof}

%%%%%%%%%%%%%%%%%%%%%%%%%%%%%%%%%%%%%%%%%%%%%%%%%%%%%%%%%%%%%%%%%%%%%%%%%%%%%%%
\flushleft{{\em Acknowledgment.} 
The authors gratefully acknowledge the support by the 
Czech Grant Agency through the grant GA CR P 201/12/G028.

We would also like to thank Toshihisa Kubo for discovering an error and suggesting the correction to an earlier version of Proposition \ref{propHMMstandard}.
}

%\bibliographystyle{plain}
%\bibliography{../TodorMilevsBibliography}

\vspace{0.3cm}

Petr Somberg

Mathematical Institute of Charles University,

Sokolovsk\'a 83, Praha 8 - Karl\'{\i}n, Czech Republic, 

E-mail: somberg@karlin.mff.cuni.cz.

\vspace{0.3cm}

Todor Milev

Department of Mathematics, University of Massachusetts Boston

100 William T. Morrissey Boulevard Boston, MA 02125, USA

E-mail: todor.milev@gmail.com

%%%%%%%%%%%%%%%%%%%%%%%%%%%%%%%%%%%%%%%%%%%%%%%%%%%%%%%%%%%%%%%%%%%%%%%%%%%%%%%

\end{document}